\documentclass[a4paper,10pt]{amsart}

\usepackage[utf8]{inputenc}
\usepackage[english]{babel}
\usepackage[T1]{fontenc}
\usepackage{amsfonts}
\usepackage{amsmath}
\usepackage{amsthm}
\usepackage{amssymb}
\usepackage{hyperref}
\usepackage{syntonly}
\usepackage{mathtools}
\usepackage{algorithm}
\usepackage{algpseudocode}
\usepackage{tikz}
\usepackage{dsfont}
\usepackage{braket}
\usepackage{faktor}
\usepackage{array}
\usepackage{lscape}
\usepackage{rotating}
\usepackage{epigraph}
\usepackage{subfig}
\usepackage{indentfirst}
\usepackage{booktabs}
\usepackage{multirow}
\usepackage{caption}
\usepackage{tabu}
\usepackage{tabularx}
\usepackage{xcolor}
\usepackage{pgf}
\usepackage{tikz}
\usepackage{caption}
\usepackage{color}
\usepackage{enumitem}
\usepackage{mathtools}
\usepackage{float}
\usepackage{graphicx}
\usepackage[displaymath, tightpage]{preview}
\usepackage[toc,page]{appendix}

\usetikzlibrary{matrix}
\usetikzlibrary{arrows}
\hypersetup{
    colorlinks,
    linkcolor={red!50!black},
    citecolor={blue!50!black},
    urlcolor={blue!80!black}
}

\usepackage[colorinlistoftodos, textwidth=4cm, shadow,disable]{todonotes}

\theoremstyle{plain}
\newtheorem{theorem}{Theorem}[section]

\newtheorem{corollary}{Corollary}[section]

\theoremstyle{definition}

\theoremstyle{remark}

\newtheorem{note}{Note}[section]

\begin{document}

\title[Centering toric arrangements of maximal rank]{Centering toric arrangements of maximal rank}

\author[Elia Saini]{Elia Saini}
\address{(Elia Saini) Independent researcher, Via San Rocco 40/A, IT-23036, Teglio (SO).}
\email{saini.elia@outlook.com}

\begin{abstract}
The homotopy type of the complement manifold of a complexified toric arrangement has been investigated by d'Antonio and Delucchi in a paper that shows the minimality of such topological space. In this work we associate to a given toric arrangement a matrix that represents the arrangement over the integers. Then, we consider the family of toric arrangements for which this matrix has maximal rank.
Our goal is to prove, by means of basic linear algebra arguments, that the complement manifold of the toric arrangements that belong to this family is diffeomorphic to that of centered toric arrangements and thus it is a minimal topological space, too.
\end{abstract}

\maketitle

 \section*{Introduction}

Toric arrangements are a well studied class of hypersurface arrangements that appear in many branches of mathematics such as algebraic geometry and topology as well as combinatorics. Following the seminal works of Lefschetz \cite{L24} and Deligne \cite{D71}, toric arrangements provide a generalization of hyperplane arrangements in the context of the study of the complements of normal crossing divisors in smooth projective varieties. A very rich area of research focuses on the description of the combinatorial and homotopic invariants of these arrangements.

The Betti numbers of the complement manifold of a toric arrangement were first computed by Looijenga in \cite{L93} using several facts from sheaf theory. Afterwards, a presentation of the complex cohomology ring has been given in \cite{DP10}.

Several authors such as Br\"{a}nd\'en, D'Adderio, Lenz and Moci explored in the series of papers
\cite{BM14}, \cite{DM13}, \cite{L17} and \cite{M12} the very rich combinatorial structure of toric arrangements and their deep relationship with arithmetic matroids. The central question of describing the homotopy type of the complement manifold of a toric arrangements was first teackled by Moci and Settepanella in \cite{MS11}. Later, d'Antonio and Delucchi in the works \cite{dD12} and \cite{dD15} exploited  the toric complex introduced in \cite{MS11}
to describe the fundamental group and
to prove a minimality result for those toric arrangements that restricts to a real arrangement. A direct consequence of this minimality result is the torsion-freeness of the integer cohomology of this class of arrangements.

In the works of Randell \cite{Ran02} and Dimca and Papadima \cite{DP03} it is proved that the complement manifold of a hyperplane arrangement is always minimal. A similar result does not hold in the context of arbitrary hypersurfaces. Indeed, it is enough to consider the complement of the plane cusp as shown in \cite{Ran02}. Thus, the interest in studying the homotopy type of the complement manifold of a toric arrangement arises in the quest for a better understanding of those hypersurface arrangements for which these minimality results hold.

Finally, improving on the study of the toric complex and the poset of layers, Callegaro and Delucchi exhibited in \cite{CD17}
a presentation of the integer cohomology of the complement manifold of centered toric arrangements. In order to describe those algebraic and topological invariants that are combinatorially determined, Pagaria provided in \cite{Pag19} an example of two toric arrangements with the same poset of layers but with different integer cohomologies, while De Concini and Gaiffi showed in \cite{DCG21} that the rational cohomology ring of toric arrangements is combinatorial.

In this work we provide a description of a wider class of toric arrangements for which the minimality results previously described hold. This class consists of those toric arrangements that have an associated matrix of maximal rank. Using techniques from basic linear algebra, we are going to build a sequence of diffeomorphic changes of coordinates that transform these arrengements into centered ones, enabling us to apply the results on complexified arrangements developed by d'Antonio and Delucchi in \cite{dD15}. In particular, this implies that the integer cohomology of any of these toric arrangements is torsion-free.
As far as we know, examples of toric arrangements with non-minimal complement manifold are not available in literature.

    \smallskip
   \noindent \textbf{Overview.} Section \ref{Section1} provides some basic definitions and results on toric arrangements as well as complexified and centered arrangements, while Section \ref{Section2} is devoted to the proof of our result on toric arrangements that have an associated matrix of maximal rank.

\section{Basics}\label{Section1}
The purpose of this section is to recall some basic definitions and results about toric arrangements and complexified arrangements, in order to fix notations and set up the frame for the subsequent part of this work.
For a general theory of toric arrangements we refer to the paper \cite{DP10} and the survey \cite{DP11}, while for a detailed study of the homotopy type of the complement manifold of complexified toric arrangements we point to the article \cite{dD15}.

Moreover, to better understand the inner and surprising connections between toric arrangements, partition functions and box splines we point out the work of De Concini and Procesi \cite{DP11} where all these topics are summarized.

\subsection{Toric arrangements}\label{TA}
A \textit{toric arrangement} in $(\mathbb{C}^{\ast})^{n}$ is a finite collection $\mathcal{A}=\{H_{1},\ldots,H_{m}\}$ of subspaces of
$(\mathbb{C}^{\ast})^{n}$, called \textit{subtori}, of the form
\begin{equation*}
 H_{i}=\left\lbrace
        (z_{1},\ldots,z_{n})\in(\mathbb{C}^{\ast})^{n}
        \left|
         z_{1}^{p_{i,1}}\cdots z_{n}^{p_{i,n}}=\alpha_{i}
        \right.
       \right\rbrace
\end{equation*}
\noindent where $p_{i,j}\in\mathbb{Z}$ and $\alpha_{i}\in\mathbb{C}^{\ast}$ as $1\leq i\leq m$ and $1\leq j\leq n$. The \textit{complement manifold} $M(\mathcal{A})$ is the complement of the union of the subtori $H_{i}$ in $(\mathbb{C}^{\ast})^{n}$, that is, the topological space
\begin{equation*}
 M(\mathcal{A})=(\mathbb{C}^{\ast})^{n}\setminus\bigcup_{i=1}^{m}H_{i}
\end{equation*}
\noindent The matrix \textit{associated} to the toric arrangement $\mathcal{A}$ is the integer matrix  of $m$ rows and $n$ columns $\tilde{M}_{\mathcal{A}}\in M_{m,n}(\mathbb{Z})$ defined by setting
\begin{equation*}
 \tilde{M}_{\mathcal{A}}=\left(
                  \begin{array}{ccc}
                   p_{1,1} & \cdots & p_{1,n} \\
                   \vdots  & \vdots & \vdots  \\
                   p_{m,1} & \cdots & p_{m,n} \\
                  \end{array}
                 \right)
\end{equation*}

\begin{note}\label{UNIMOD} A square integer matrix is \textit{unimodular} if its determinant equals $1$ or $-1$. In this case, a basic result in linear algebra ensures us that the inverse matrix still has integer entries.
\end{note}

\subsection{Complexified arrangements}\label{CX}
A toric arrangement $\mathcal{A}=\{H_{1},\ldots,H_{m}\}$ in $(\mathbb{C}^{\ast})^{n}$ is \textit{complexified} if it restricts to a real arrangement, that is, if $\alpha_{i}\in S^{1}$ for all $i=1,\ldots,m$ while it is \textit{centered} if all the $\alpha_{i}\text{'s}$ equal $1$. Clearly a centered toric arrangement is complexified, too.
A topological space $X$ is \textit{minimal} if it is homotopically equivalent to a CW complex with exactly $b_{k}$ cells in dimension $k$, where $b_{k}$ denotes the $k\text{-th}$ Betti number of $X$. The minimality of the complement manifold of a complexified toric arrangement has been proved by d'Antonio and Delucchi in the following theorem, that is the basement on which our work is funded.

\begin{theorem}[\cite{dD15}, Corollay 6.10]\label{BASEMENT}
 Let $\mathcal{A}=\{H_{1},\ldots,H_{m}\}$ be a toric arrangement in
 $(\mathbb{C}^{\ast})^{n}$. If $\mathcal{A}$ is complexified, then the complement manifold $M(\mathcal{A})$ is a minimal topological space.
\end{theorem}

\section{Results}\label{Section2}
The aim of this section is to prove the minimality of the complement manifold of a toric arrangement with associated matrix that has maximal rank. Our goal is to extend to this wider class of toric arrangements the ideas described in \cite[Section 7.2]{CD17} also presented by Pagaria in a series of conference talks. To do this we are going to perform a sequence of changes of coordinates that transform the given toric arrangement into a centered one (compare Section \ref{CX}) so that we are allowed to apply Theorem \ref{BASEMENT}.

\begin{theorem}\label{main}
 Let $n\geq m\geq1$ and let $\mathcal{A}=\{H_{1},\ldots,H_{m}\}$ be a toric arrangement in $(\mathbb{C}^{\ast})^{n}$ with associated matrix
 $\tilde{M}_{\mathcal{A}}$. If $\tilde{M}_{\mathcal{A}}$ has rank $m$, then there exists a centered toric arrangement $\mathcal{B}=\{T_{1},\ldots,T_{m}\}$ in $(\mathbb{C}^{\ast})^{n}$ such that the complement manifolds $M(\mathcal{A})$ and $M(\mathcal{B})$ are diffeomorphic.
\end{theorem}

Thus, Theorem \ref{BASEMENT} implies that the complement manifold of these toric arrangements is a minimal topological space.

\begin{corollary}
 Let $n\geq m\geq1$ and let $\mathcal{A}=\{H_{1},\ldots,H_{m}\}$ be a toric arrangement in $(\mathbb{C}^{\ast})^{n}$ with associated matrix
 $\tilde{M}_{\mathcal{A}}$. If $\tilde{M}_{\mathcal{A}}$ has rank $m$, then the
 complement manifold $M(\mathcal{A})$ is a minimal topological space.
\end{corollary}

Moreover, with the same arguments of \cite[Corollary 14]{dD15} it follows that the complement manifold of these toric arrangements is torsion-free, too.

\begin{corollary}
  Let $n\geq m\geq1$ and let $\mathcal{A}=\{H_{1},\ldots,H_{m}\}$ be a toric arrangement in $(\mathbb{C}^{\ast})^{n}$ with associated matrix
  $\tilde{M}_{\mathcal{A}}$. If $\tilde{M}_{\mathcal{A}}$ has rank $m$, then the cohomology groups $H^{k}(M(\mathcal{A}),\mathbb{Z})$ are torsion-free.
\end{corollary}

\begin{proof}[Proof of Theorem \ref{main}]
\textit{Step 1.} Up to permuting coordinates in $(\mathbb{C}^{\ast})^{n}$ we can assume that the first $m$ columns of the matrix $\tilde{M}_{\mathcal{A}}$ form a $m\times m$ minor $M_{\mathcal{A}}$ of maximal rank.
Combining Gaussian elimination and Euclidean integer division, \cite[Theorem II.2]{New72} ensures us that there exist a unimodular matrix $H_{\mathcal{A}}$ such that $M_{\mathcal{A}}H_{\mathcal{A}}$ is an upper triangular matrix.
Since $H_{\mathcal{A}}\in M_{m,m}(\mathbb{Z})$ is unimodular, from Note \ref{UNIMOD} we know that its inverse  $K_{\mathcal{A}}$ still has integer entries. Let us denote the entries of $H_{\mathcal{A}}$ and $K_{\mathcal{A}}$ by $h_{i,j}$ and $k_{i,j}$, respectively. Here, $i$ stands for the row index and $j$ for the column index.

\noindent \textit{Step 2.} Let $z_{1},\ldots,z_{n}$ be coordinates in
$(\mathbb{C}^{\ast})^{n}$. For $1\leq i\leq n$ set
\begin{equation}\label{E1}
 t_{i}=
 \left\lbrace
 \begin{array}{lll}
  z_{1}^{k_{i,1}}\cdots z_{m}^{k_{i,m}} & \text{if} & 1\leq i\leq m \\
  z_{i}                                           & \text{if} & m+1\leq i\leq n \\
 \end{array}
 \right.
\end{equation}
This defines a change of coordinates $(z_{1},\ldots,z_{n})\mapsto(t_{1},\ldots,t_{n})$ in $(\mathbb{C}^{\ast})^{n}$ with inverse given componentwise by
\begin{equation*}
 z_{i}=
\left\lbrace
\begin{array}{lll}
 t_{1}^{h_{i,1}}\cdots t_{m}^{h_{i,m}} & \text{if} & 1\leq i\leq m \\
 t_{i}                                           & \text{if} & m+1\leq i \leq n \\
\end{array}
\right.
\end{equation*}
The change of coordinates defined in \eqref{E1} transforms the given toric arrangement
$\mathcal{A}=\{H_{1},\ldots,H_{m}\}$ with subtori
\begin{equation*}
 H_{i}=\left\lbrace
        (z_{1},\ldots,z_{n})\in(\mathbb{C}^{\ast})^{n}\left|
                                                       z_{1}^{p_{i,1}}\cdots z_{n}^{p_{i,n}}=\alpha_{i}
                                                      \right.
       \right\rbrace
\end{equation*}
into the toric arrangement $\mathcal{C}=\{K_{1},\ldots,K_{m}\}$ with subtori
\begin{equation*}
 K_{i}=\left\lbrace
        (t_{1},\ldots,t_{n})\in(\mathbb{C}^{\ast})^{n}
        \left|
        t_{i}^{d_{i}}t_{i+1}^{a_{i,i+1}}\cdots t_{m}^{a_{i,m}}t_{m+1}^{p_{i,m+1}}\cdots t_{n}^{p_{i,n}}=\alpha_{i}
       \right.
       \right\rbrace
\end{equation*}
To see this, let $(z_{1},\ldots,z_{n})$ be a point of $(\mathbb{C}^{\ast})^{n}$ that belongs to the subtorus $H_{i}$. Therefore, its coordinates fulfill condition
\begin{equation*}
 z_{1}^{p_{i,1}}\cdots z_{n}^{p_{i,n}}=\alpha_{i}
\end{equation*}
Exploiting the change of coordinates given by \eqref{E1} the previous condition is then equivalent to
\begin{equation*}
 t_{i}^{d_{i}}t_{i+1}^{a_{i,i+1}}\cdots t_{m}^{a_{i,m}}t_{m+1}^{p_{i,m+1}}\cdots t_{n}^{p_{i,n}}=\alpha_{i}
\end{equation*}
where $a_{i,s}=p_{i,1}k_{1,s}+\ldots+p_{i,m}k_{m,s}$ for $1\leq i\leq m$ and $i+1\leq s\leq m$. Indeed, since from the first step of our proof
\begin{equation*}
 M_{\mathcal{A}}H_{\mathcal{A}}=\left(
                         \begin{array}{ccccc}
                          d_{1} & a_{1,2} & \cdots & \cdots & a_{1,m} \\
                          \quad & d_{2}   & a_{2,3} & \cdots & a_{2,m} \\
                          \quad & \quad & \ddots  & \quad & \vdots    \\
                          \quad & \quad & \quad & d_{m-1} & a_{m-1,m} \\
                          \quad & \quad & \quad & \quad & d_{m}       \\
                         \end{array}
                        \right)
\end{equation*}
is an upper triangular matrix, we have
\begin{equation*}
 \begin{array}{ll}
  z_{1}^{p_{i,1}}\cdots z_{n}^{p_{i,n}}=\alpha_{i} & \Longleftrightarrow \\
  \left(t_{1}^{h_{1,1}}\cdots t_{m}^{h_{1,m}}\right)^{p_{i,1}}\cdots\left(t_{1}^{h_{m,1}}\cdots t_{m}^{h_{m,m}}\right)^{p_{i,m}}t_{m+1}^{p_{i,m+1}}\cdots t_{n}^{p_{i,n}}=\alpha_{i} & \Longleftrightarrow \\
  t_{1}^{p_{i,1}h_{1,1}+\ldots+p_{i,m}h_{m,1}}\cdots t_{m}^{p_{i,1}h_{1,m}+\ldots+p_{i,m}h_{m,m}} t_{m+1}^{p_{i,m+1}}\cdots t_{n}^{p_{i,n}}=\alpha_{i} & \Longleftrightarrow \\
  t_{1}^{(M_{\mathcal{A}})_{i}(H_{\mathcal{A}})^{1}}\cdots
  t_{m}^{(M_{\mathcal{A}})_{i}(H_{\mathcal{A}})^{m}}
  t_{m+1}^{p_{i,m+1}}\cdots t_{n}^{p_{i,n}}=\alpha_{i} & \Longleftrightarrow \\
  t_{i}^{d_{i}}t_{i+1}^{a_{i,i+1}}\cdots t_{m}^{a_{i,m}}t_{m+1}^{p_{i,m+1}}\cdots t_{n}^{p_{i,n}}=\alpha_{i} & \quad \\
 \end{array}
\end{equation*}
Here $(M_{\mathcal{A}})_{i}$ and
$(H_{\mathcal{A}})^{l}$ stand for the $i\text{-th}$ row and the $l\text{-th}$ column of the matrices
$M_{\mathcal{A}}$ and $H_{\mathcal{A}}$, respectively.

\noindent \textit{Step3.} Let us consider the toric arrangement $\mathcal{C}=\{K_{1},\ldots,K_{m}\}$ and let $t_{1},\ldots,t_{m}$ be coordinates in $(\mathbb{C}^{\ast})^{n}$. For $1\leq i\leq n$ set
\begin{equation}\label{E2}
    t_{(1)i}=\left\lbrace
              \begin{array}{lll}
               t_{i} & \text{if} & i\neq m \\
               t_{m}\gamma_{m}^{-1} & \text{if} & i=m \\
              \end{array}
             \right.
\end{equation}
where $\gamma_{m}$ is a $d_{m}\text{-th}$ root of $\alpha_{m}$. Thus,
$(t_{1},\ldots,t_{n})\mapsto(t_{(1)1},\ldots,t_{(1)n})$ defines a change of coordinates in $(\mathbb{C}^{\ast})^{n}$ with inverse given componentwise by
\begin{equation*}
    t_{i}=\left\lbrace
              \begin{array}{lll}
               t_{(1)i} & \text{if} & i\neq m \\
               t_{(1)m}\gamma_{m} & \text{if} & i=m \\
              \end{array}
             \right.
\end{equation*}
that transforms the toric arrangement $\mathcal{C}=\{K_{1},\ldots,K_{m}\}$ into the toric arrangement $\mathcal{C}_{(1)}=\{K_{(1)1},\ldots,K_{(1)m}\}$ where $K_{(1)i}$ are subtori of the form
\begin{equation*}
     K_{(1)i}=
       \left\lbrace
        t_{(1)i}^{d_{i}}t_{(1)i+1}^{a_{i,i+1}}\cdots t_{(1)m}^{a_{i,m}}t_{(1)m+1}^{p_{i,m+1}}\cdots t_{(1)n}^{p_{i,n}}=\alpha_{(1)i}
       \right\rbrace
\end{equation*}
with
\begin{equation*}
    \alpha_{(1)i}=\left\lbrace
                    \begin{array}{lll}
                     \alpha_{i}\gamma_{m}^{-a_{i,m}} & \text{if} & 1\leq i\leq m-1 \\
                     1 & \text{if} & i=m \\
                    \end{array}
                  \right.
\end{equation*}
In order to prove this, let $(t_{1},\ldots,t_{n})$ be a point of $(\mathbb{C}^{\ast})^{n}$ that belongs to the subtorus $K_{i}$. Then, its coordinates satisfy condition
\begin{equation*}
     t_{i}^{d_{i}}t_{i+1}^{a_{i,i+1}}\cdots t_{m}^{a_{i,m}}t_{m+1}^{p_{i,m+1}}\cdots t_{n}^{p_{i,n}}=\alpha_{i}
\end{equation*}
Thanks to the change of coordinates defined by \eqref{E2} this is equivalent to
\begin{equation*}
        t_{(1)i}^{d_{i}}t_{(1)i+1}^{a_{i,i+1}}\cdots t_{(1)m}^{a_{i,m}}t_{(1)m+1}^{p_{i,m+1}}\cdots t_{(1)n}^{p_{i,n}}=\alpha_{(1)i}
\end{equation*}
To see this we have to distinguish between two cases. If $1\leq i\leq m-1$ we have
\begin{equation*}
 \begin{array}{ll}
          t_{i}^{d_{i}}t_{i+1}^{a_{i,i+1}}\cdots t_{m}^{a_{i,m}}t_{m+1}^{p_{i,m+1}}\cdots t_{n}^{p_{i,n}}=\alpha_{i} & \Longleftrightarrow \\
          t_{(1)i}^{d_{i}}t_{(1)i+1}^{a_{i,i+1}}\cdots t_{(1)m-1}^{a_{i,m-1}}(t_{(1)m}\gamma_{m})^{a_{i,m}}t_{(1)m+1}^{p_{i,m+1}}\cdots t_{(1)n}^{p_{i,n}}=\alpha_{i} & \Longleftrightarrow \\
          t_{(1)i}^{d_{i}}t_{(1)i+1}^{a_{i,i+1}}\cdots t_{(1)m}^{a_{i,m}}t_{(1)m+1}^{p_{i,m+1}}\cdots t_{(1)n}^{p_{i,n}}=\alpha_{i}\gamma_{m}^{-a_{i,m}} & \Longleftrightarrow \\
          t_{(1)i}^{d_{i}}t_{(1)i+1}^{a_{i,i+1}}\cdots t_{(1)m}^{a_{i,m}}t_{(1)m+1}^{p_{i,m+1}}\cdots t_{(1)n}^{p_{i,n}}=\alpha_{(1)i} & \quad \\
 \end{array}
\end{equation*}
and similarly, if $i=m$ we have
\begin{equation*}
 \begin{array}{ll}
          t_{m}^{d_{m}}t_{m+1}^{p_{i,m+1}}\cdots t_{n}^{p_{i,n}}=\alpha_{m} & \Longleftrightarrow \\
          (t_{(1)m}\gamma_{m})^{d_{m}}t_{(1)m+1}^{p_{i,m+1}}\cdots t_{(1)n}^{p_{i,n}}=\alpha_{m} & \Longleftrightarrow \\
          t_{(1)m}^{d_{m}}t_{(1)m+1}^{p_{i,m+1}}\cdots t_{(1)n}^{p_{i,n}}=\frac{\alpha_{m}}{\alpha_{m}} & \Longleftrightarrow \\
          t_{(1)m}^{d_{m}}t_{(1)m+1}^{p_{i,m+1}}\cdots t_{(1)n}^{p_{i,n}}=1 & \Longleftrightarrow \\
          t_{(1)m}^{d_{m}}t_{(1)m+1}^{p_{i,m+1}}\cdots t_{(1)n}^{p_{i,n}}=\alpha_{(1)m} & \quad \\
 \end{array}
\end{equation*}
Exploiting the triangular structure of the matrix
$M_{\mathcal{A}}H_{\mathcal{A}}$, for $k=2,\ldots,m$ we can recursively build a sequence  $(t_{(k-1)1},\ldots,t_{(k-1)n})\mapsto(t_{(k)1},\ldots,t_{(k)n})$
of changes of coordinates in $(\mathbb{C}^{\ast})^{n}$ defined componentwise by
\begin{equation*}
    t_{(k)i}=\left\lbrace
              \begin{array}{lll}
               t_{(k-1)i} & \text{if} & i\neq m-k+1 \\
               t_{(k-1)m-k+1}\gamma_{m-k+1}^{-1} & \text{if} & i=m-k+1 \\
              \end{array}
             \right.
\end{equation*}
that transform tha toric arrangement $\mathcal{C}_{(k-1)}=\{K_{(k-1)1},\ldots,K_{(k-1)m}\}$ into the toric arrangement $\mathcal{C}_{(k)}=\{K_{(k)1},\ldots,K_{(k)m}\}$ where $K_{(k)i}$ are subtori of the form
\begin{equation*}
     K_{(k)i}=
       \left\lbrace
        t_{(k)i}^{d_{i}}t_{(k)i+1}^{a_{i,i+1}}\cdots t_{(k)m}^{a_{i,m}}t_{(k)m+1}^{p_{i,m+1}}\cdots t_{(k)n}^{p_{i,n}}=\alpha_{(k)i}
       \right\rbrace
\end{equation*}
with
\begin{equation*}
 \alpha_{(k)i}=
 \left\lbrace
 \begin{array}{llll}
  \alpha_{(k-1)i}\gamma_{m-k+1}^{-a_{i,m-k+1}} & \text{if} & 1\leq i\leq m-k+1 \\
  1 & \text{if} & m-k+1\leq i\leq m \\
 \end{array}
 \right.
\end{equation*}
Again, here
$\gamma_{m-k+1}$ is a $d_{m-k+1}\text{-th}$ root of $\alpha_{(k-1)m-k+1}$.

\noindent \textit{Step 4.} The toric arrangement $\mathcal{C}_{(m)}$ is then centered (compare Section \ref{CX}). The sequence of diffeomorphic changes of coordinates in $(\mathbb{C}^{\ast})^{n}$ that map the toric arrangement $\mathcal{A}$ into the toric arrangement $\mathcal{C}_{(m)}$ induces by restriction a diffeomorphism between the complement manifolds $M(\mathcal{A})$ and $M(\mathcal{C}_{(m)})$.
Hence, placing $T_{i}=K_{(m)i}$ as $1\leq i\leq m$ and setting
$\mathcal{B}=\{T_{1},\ldots,T_{m}\}$, our statement follows.

\begin{note}
 The diffeomorphisms $(t_{(k-1)1},\ldots,t_{(k-1)n})\mapsto(t_{(k)1},\ldots,t_{(k)1})$ defined in the third step of the proof of Theorem \ref{main} are not canonical. Indeed, they depend on the recursive choice of the complex roots $\gamma_{m},\ldots,\gamma_{1}$. However, the diffeomorphism type of the complement manifold $M(\mathcal{A})$ is uniquely determined.
\end{note}

\end{proof}

%

%

\smallskip
\noindent \textbf{Acknowledgments.} The author would like to mention
Farhad Babaee for suggesting the idea behind the third step of Theorem \ref{main} during a conversation that took place at the University of Fribourg (CH) in January, 2016. Moreover, the author would express his gratitude to Emanuele Delucchi and Roberto Pagaria for some enlightening discussions regarding the basic properties of toric arrangements as well as to William C. Jagy and Matthias Lenz for pointing out the bibliographical references \cite{New72} and \cite{L17}, respectively. The author would also like to express a special thanks to Alberto Cavallo for the very helpful support provided throughout
all the time this project has been developed and to Isacco Saini for some tips concerning the title of the paper.
Finally the author would like to thank the anonymous referees for their patient work and their very
useful suggestions.

\bibliographystyle{amsalpha}

\bibliography{Full_revised}

\end{document}